\documentclass[12pt,reqno]{amsart}

\usepackage[T1]{fontenc}
\usepackage{ae,aecompl}
\usepackage{graphics,cancel,graphicx}
\usepackage{epsfig,psfrag,manfnt}
\usepackage{mathrsfs}
\usepackage{graphicx,mathabx}
\usepackage{bbm,subfigure,rotating,bm,float,mathdots,wasysym,scalerel}

\usepackage[all,cmtip]{xy}
\usepackage{amssymb,amsmath,amsfonts,amsthm,color}

\usepackage{scalerel,stackengine}
\stackMath
\newcommand\reallywidehat[1]{%
\savestack{\tmpbox}{\stretchto{%
  \scaleto{%
    \scalerel*[\widthof{\ensuremath{#1}}]{\kern-.6pt\bigwedge\kern-.6pt}%
    {\rule[-\textheight/2]{1ex}{\textheight}}
  }{\textheight}%
}{0.5ex}}%
\stackon[1pt]{#1}{\tmpbox}%
}

\newcommand{\calA}{\mathcal{A}}

\newcommand{\calL}{\mathcal{L}}
\newcommand{\calM}{\mathcal{M}}

\newcommand{\calO}{\mathcal{O}}
\newcommand{\calP}{\mathcal{P}}

\newcommand{\calT}{\mathcal{T}}

\newcommand{\calW}{\mathcal{W}}

\newcommand{\mC}{\mathbb{C}}
\newcommand{\mD}{\mathbb{D}}

\newcommand{\mN}{\mathbb{N}}

\newcommand{\mQ}{\mathbb{Q}}
\newcommand{\mR}{\mathbb{R}}

\newcommand{\mU}{\mathbb{U}}

\newcommand{\mZ}{\mathbb{Z}}

\newcommand{\bbe}{\mathbf{e}}

\newcommand{\bbn}{\mathbf{n}}

\newcommand{\bbv}{\mathbf{v}}

\newcommand{\bbz}{\mathbf{z}}

\newcommand{\SL}{{\textrm{SL}}}
\newcommand{\eSL}{{\textrm{\em SL}}}
\newcommand{\E}{{\textrm{E}}}
\newcommand{\eE}{{\textrm{\em E}}}

\newtheorem{theorem}{Theorem}[section]
\newtheorem{lemma}[theorem]{Lemma}
\newtheorem{corollary}[theorem]{Corollary}
\newtheorem{proposition}[theorem]{Proposition}

\theoremstyle{definition}

\newtheorem{remark}[theorem]{Remark}

\theoremstyle{definition}
\newtheorem{definition}[theorem]{Definition}

\theoremstyle{definition}

\theoremstyle{definition}

\begin{document}

\keywords{Convolution Banach algebras of Borel measures, factorisation by elementary matrices, commutative unital Banach algebras, $K$-theory}

\subjclass[2010]{Primary 15A23; Secondary 43A10, 46J10, 20H25}

 \title[]{
 Generation of the special linear group \\
 by elementary matrices \\
 in some measure Banach algebras}

\author[]{Amol Sasane}
\address{Department of Mathematics \\London School of Economics\\
     Houghton Street\\ London WC2A 2AE\\ United Kingdom}
\email{A.J.Sasane@lse.ac.uk}

 \maketitle
 
 \vspace{-1cm}
 
 \begin{abstract}
 For a commutative unital ring $R$, and $n\in \mN$, let $\SL_n(R)$ denote the special linear group over $R$, and 
 $\E_n(R)$ the subgroup of elementary matrices. Let $\calM^+$ be the Banach algebra of all complex Borel measures on $[0,+\infty)$ with the norm given by the total variation, the usual operations of addition and scalar multiplication, and with convolution. It is first shown that $\SL_n(A)=\E_n(A)$ for Banach subalgebras $A$ of $\calM^+$ that are closed under the operation $\calM^+\owns \mu \mapsto \mu_t$, $t\in [0,1]$, where 
 $\mu_t(E):=\int_E (1-t)^x d\mu(x)$ for $t\in [0,1)$, and Borel subsets $E$ of $[0,+\infty)$,  and $\mu_1:=\mu(\{0\})\delta$, where $\delta\in \calM^+$ is the Dirac measure. Using this, and with auxiliary results established in the article,   many illustrative examples of such Banach algebras $A$ are given, including several well-studied classical Banach algebras such as the class of analytic almost periodic functions. An example of a Banach subalgebra $A\subset \calM^+$, that does not possess the closure property above, but for which $\SL_n(A)=\E_n(A)$ nevertheless holds,  is also constructed. 
 \end{abstract}
 
  \vspace{-0.39cm}
 
 \tableofcontents
 
 \section{Introduction}

\noindent The aim of this article is to establish $\SL_n(R)=\E_n(R)$, where $\SL_n(R)$ is the special linear group over $R$, $\E_n(R)$ is the group of elementary matrices, and $R$ is a certain Banach algebra of measures on the half-line $[0,+\infty)$. We elaborate on this below. 

\begin{definition}[$\SL_n(R)$ and $\E_n(R)$]$\;$

\noindent Let $R$ be a commutative unital ring with multiplicative identity $1$ and additive identity element $0$. Let $n\in \mN=\{1,2,3,\cdots\}$. 

\smallskip 

\noindent $\bullet$ Let $R^{\;\!n\times n} $ denote the set of matrices over $R$ with $n$ rows and $n$ 

\noindent {\phantom{$\bullet$} }columns.  We denote by $I_n$ the identity matrix in $R^{\;\!n\times n}$, i.e., 

\noindent {\phantom{$\bullet$} }the matrix with all diagonal entries equal to $1\in R$ and off-diagonal 

\noindent {\phantom{$\bullet$} }entries $0\in R$. 

\smallskip 

\goodbreak

\noindent $\bullet$ The {\em special linear group} $\SL_n(R)$ denotes the group (with matrix 

\noindent {\phantom{$\bullet$} }multiplication) of all matrices $M\in R^{\;\!n\times n} $ whose entries belong to $R$ 

\noindent {\phantom{$\bullet$} }and the determinant $\det M$ of $M$ is $1$. The {\em general linear group} 

\noindent {\phantom{$\bullet$} }$\textrm{GL}_n(R)$ consists of all invertible matrices in  $R^{\;\!n\times n}$. 

\smallskip 

\noindent $\bullet$ An {\em elementary matrix} 
$E_{ij}(\alpha)$ over $R$ has the form $E_{ij}=I_n+\alpha \;\!\bbe_{ij}$, 

\noindent {\phantom{$\bullet$} }where $i\neq j$, $\alpha \in R$, and $\bbe_{ij}$ is the $n\times n$ matrix whose entry in the 

\noindent {\phantom{$\bullet$} }$i$th row and $j$th column is $1$, and all the other entries of $\bbe_{ij}$ are zeros. 

\smallskip 

\noindent $\bullet$ $\E_n(R)$ is the subgroup of $\SL_n(R)$ generated by the elementary 

\noindent {\phantom{$\bullet$} }matrices. 
\end{definition}

\noindent A classical question in algebra is: 
$$
(\mathbf{Q}) \textrm{ For all }n\in \mN\textrm{, is }\SL_n(R)=\E_n(R)\textrm{?} 
$$
The answer depends on the ring $R$. For example:

\smallskip 

\noindent $\bullet$ If $R=\mC$, then the answer to ($\mathbf{Q}$) is `Yes', and is an exercise in linear 

\noindent {\phantom{$\bullet$} }algebra; see for instance \cite[Chap.2, \S2, Exercise~18(c), p.71]{Art}. 

\smallskip 

\noindent $\bullet$ Let $R=\mC[z_1,\cdots, z_d]$. 

\noindent {\phantom{$\bullet$} }If $d=1$, then the answer  to ($\mathbf{Q}$)  is `Yes': This follows from the 

\noindent {\phantom{$\bullet$} }Euclidean division algorithm in $\mC[z]$. 

\noindent {\phantom{$\bullet$} }If $d=2$, then the answer  to ($\mathbf{Q}$) is 
`No': A counterexample is (\cite{Coh}):
$$
\left[ \begin{array}{cc} 1+z_1z_2 & z_1^2 \\ -z_2^2 & 1-z_1 z_2 \end{array}\right] 
\in \SL_2(\mC[z_1,z_2])\setminus E_2(\mC[z_1,z_2]).
$$
\noindent {\phantom{$\bullet$} }For $d\geq 3$, the answer  to ($\mathbf{Q}$) is `Yes': This is the $K_1$-analogue of 

\noindent {\phantom{$\bullet$} }Serre's conjecture, which is the Suslin stability theorem \cite{Sus}. 

\smallskip 

\noindent $\bullet$ Rings of continuous real- or complex-valued functions on a topological 

\noindent {\phantom{$\bullet$} }space $X$ were considered in \cite{Vas}. 

\smallskip 

\goodbreak

\noindent $\bullet$ For the ring $\calO_X$ of holomorphic functions on Stein spaces in $\mC^d$, 
the 

\noindent {\phantom{$\bullet$} }question ($\mathbf{Q}$) was posed as an explicit open problem by Gromov in 

\noindent {\phantom{$\bullet$} }\cite{Gro}, and was solved in \cite{IK1}.  See also \cite{IK0} and \cite{IK2}.

\smallskip 

\noindent $\bullet$ The question  ($\mathbf{Q}$) for some rings of `sequences' (with termwise 

\noindent {\phantom{$\bullet$} }operations) 
such as 

\vspace{0.1cm}

\phantom{AAAAAA}$\displaystyle
\begin{array}{ll}
\ell^\infty(\mZ^d)&\textrm{(bounded)}, \\
c_0(\mZ^d) & \textrm{(converging to }0\textrm{), and}\\
 s'(\mZ^d) & \textrm{(at most polynomially growing),}
\end{array}
$

\vspace{0.1cm}

\noindent {\phantom{$\bullet$} }were considered in \cite{Sas}. The ring $s'(\mZ^d)$ is isomorphic to the ring of 

\noindent {\phantom{$\bullet$} }periodic distributions on $\mR^d$ 
with the multiplication operation given 

\noindent {\phantom{$\bullet$} }by convolution.

\smallskip

\noindent $\bullet$ In \cite{DK}, it was shown that for a unital commutative ring $R$, if $n\geq 2$ 

\noindent {\phantom{$\bullet$} }and if the Bass stable rank of $R$ is $1$, then the answer to ($\mathbf{Q}$) is `Yes'. 

\noindent {\phantom{$\bullet$} }Using this result, and a result from \cite[\S7]{Mil} (Proposition~\ref{prop_11_may_2021_18:36} below), 

\noindent {\phantom{$\bullet$} }many examples of Banach algebras  
of holomorphic functions in one 

\noindent {\phantom{$\bullet$} }and several variables, like the polydisc/ball algebras, were considered 

\noindent {\phantom{$\bullet$} }for which the answer to ($\mathbf{Q}$) is `Yes'. 
 
 \smallskip 
 
 \noindent In this article, we will consider certain Banach algebras of measures. Our main tool will 
 be the known result  stated as Proposition~\ref{prop_11_may_2021_18:36} below (see \cite[\S7]{Mil}). Before stating this result, we elaborate on the Banach algebra structure of $A^{n\times n}$ for a Banach algebra $A$.
 
 Let $(A,\|\cdot\|)$  be a commutative unital Banach algebra. 
 Then $A^{n\times n} $ is a complex algebra with the usual matrix operations. 
 Let $A^{n\times 1}$ denote the vector space of all column vectors of size $n$ with entries from $A$ 
 and componentwise operations. Then $A^{n\times 1}$ is a normed space with the `Euclidean norm' 
 defined by $\displaystyle 
 \|\bbv\|_2^2 :=\|v_1\|^2+\cdots+\|v_n\|^2$ for all $\bbv$ in $A^{n\times 1}$, where 
  $\bbv$ has components denoted by $v_1,\cdots,v_n\in A$.  If $M\in A^{n\times n}$, 
 then the matrix multiplication map, 
  $
 A^{n\times 1}\owns \bbv \mapsto M\bbv \in A^{n\times 1}, 
 $  
 is 
 a continuous linear transformation, and we equip $A^{n\times n}$ with the induced operator norm, denoted by $\|\cdot\|$ again. Then $A^{n\times n}$ with this operator norm is a unital Banach algebra. Subsets of $A^{n\times n}$ are given the induced subspace topology. We also state the following observation which will be used later. 
 
 \begin{lemma}
 \label{lemma_14_june_2021_10:24}
 Let $(A,\|\cdot\|)$  be a commutative unital Banach algebra, and 
  $M=[m_{ij}]\in A^{n\times n}$, where the entry in the $i^{\textrm{\em th}}$ row and $j^{\textrm{\em th}}$ column of $M$ is denoted by $m_{ij}$, $1\leq i,j\leq n$. Then 
 $$
 \|M\|^2\leq \displaystyle \sum_{i=1}^n \sum_{j=1}^n \|m_{ij}\|^2.
 $$
 \end{lemma}
 
 \goodbreak 
 
 \begin{proof}  Let $\bbv\in A^{n\times 1}$ have components $v_1,\cdots, v_n\in A$. Using the Cauchy-Schwarz inequality in $\mR^n$, we have 
 \begin{eqnarray*}
 \|M\bbv \|_2^2=\sum_{i=1}^n  \Big\| \sum_{j=1}^n m_{ij} v_j\Big\|^2
 \!\!\!\!&\leq&\!\!\!\! \sum_{i=1}^n \Big( \sum_{j=1}^n \|m_{ij}v_j\|\Big)^2
 \leq \sum_{i=1}^n \Big( \sum_{j=1}^n \|m_{ij}\|\|v_j\|\Big)^2
 \\
 \!\!\!\!&\leq&\!\!\!\! \sum_{i=1}^n\sum_{j=1}^n \|m_{ij}\|^2\sum_{k=1}^n \|v_k\|^2
 =
 \sum_{i=1}^n\sum_{j=1}^n \|m_{ij}\|^2 \|\bbv\|_2^2.
 \end{eqnarray*}
 As $\bbv\in A^{n\times 1}$ was arbitrary, it follows that $ \|M\|^2\leq \displaystyle \sum_{i=1}^n \sum_{j=1}^n \|m_{ij}\|^2$.
 \end{proof}

 \noindent We now state the result from \cite[\S7]{Mil}.

 \begin{proposition}
  \label{prop_11_may_2021_18:36} $\;$
 
 \noindent 
 Let $A$ be a commutative unital Banach algebra$,$ $n\in \mN,$ and $M\in \eSL_n(A)$. Then the following are equivalent:
 \begin{itemize}
 \item $M\in \eE_n (A)$.
 \item $M$ is null-homotopic.
 \end{itemize}
 \end{proposition}
 
\begin{definition}[Null-homotopic element of $\SL_n(A)$]$\;$

\noindent Let $A$ be a commutative unital Banach algebra, and $n\in \mN$. 
An element $M\in \SL_n(A)$ is {\em null-homotopic} if $M$ is homotopic to the identity matrix $I_n$, that is there exists a continuous map $H:[0,1]\rightarrow \SL_n(A)$ such that 
$H(0)=M$ and $H(1)=I_n$. 
\end{definition} 

\noindent The statement of our first result (Theorem~\ref{theorem_11_may_2021_20:17}) 
will be given in the following section, where we introduce the Banach algebras of complex Borel measures on $[0,+\infty)$ that we will consider. We will prove Theorem~\ref{theorem_11_may_2021_20:17} in Section~\ref{section_3}. 
In Section~\ref{section_5}, using Theorem~\ref{theorem_11_may_2021_20:17}, and with auxiliary results established there,   many illustrative examples of such Banach algebras $A$ are given, including several  classical Banach algebras such as the class of holomorphic almost periodic functions. An example of a Banach subalgebra $A\subset \calM^+$, that does not possess the closure property specified in Theorem~\ref{theorem_11_may_2021_20:17}, but for which $\SL_n(A)=\E_n(A)$ nevertheless holds,  is also constructed in Section~\ref{section_5}.


\section{Banach algebras of measures}

\noindent We recall the following classical Banach algebra $\calM^+$ of measures on a half-line; see for example \cite[\S4, pp.141-150]{HP} or \cite[Chapter 6]{Rud}.

\goodbreak

\begin{definition}[The Banach algebra $\calM^+$]$\;$

\noindent Let $\calM^+$ denote the set of all complex Borel measures on $\mR$ 
with support contained in $[0,+\infty)$. Then $\calM^+$ is a complex vector space 
with addition $+$ and scalar multiplication $\cdot$ defined as usual, and 
it becomes a complex algebra if we take the operation $\ast$ of convolution of measures as the operation of multiplication. 

Given a Borel set $E\subset [0,+\infty)$, by a {\em partition of $E$} we mean a countable collection $(E_n)_{n\in \mN}$ of Borel subsets of $E$ such that whenever $m\neq n$, we have $E_n \cap E_m=\emptyset $, and ${\scaleobj{0.81}{\bigcup\limits_{n\;\!\in\;\! \mN}}}E_n=E$. 
We define $|\mu|$ by 

\vspace{-0.1cm}

$\phantom{AAAAAAAAA}\displaystyle 
|\mu|(E)=\sup_{(E_n)_{n\;\! \in \;\!\mN}\;\!\in \;\!\calP(E)} \sum_{n=1}^\infty |\mu(E_n)|,
$

\vspace{0.1cm}

\noindent 
where $\calP(E)$ is the collection of all partitions of the Borel set $E$.  Then 
$|\mu|(E)\geq |\mu(E)|$, and $|\mu|$ is a positive measure defined on all Borel subsets of $[0,+\infty)$. 
The norm of an element $\mu\in \calM^+$ is taken as the {\em total variation measure $|\mu|$ of $[0,+\infty)$}, i.e., 
$$
\|\mu\|:=\sup_{(E_n)_{n\;\!\in \;\!\mN}\;\!\in \;\!\calP} \sum_{n=1}^\infty |\mu(E_n)|,
$$
where $\calP$ is the set of all partitions of $[0,+\infty)$. 
Then $(\calM^+, +,\cdot, \ast,\|\cdot\|)$ is a commutative unital complex Banach algebra. The unit element (i.e., identity with respect to convolution) is the {\em Dirac measure} $\delta$, given by 
$$
\delta(E)= \left\{ \begin{array}{cl} 1 & \textrm{if }\; 0\in E, \\
0 & \textrm{if }\; 0\not\in E. \end{array} \right.
$$
\end{definition} 

\begin{definition}
\label{definition_14_may_2021_20:29}
If $\mu \in \calM^+$ and $t\in [0,1)$, define $\mu_t\in \calM^+$ by 
$$
\mu_t(E):=\int_E (1-t)^x \;\! d\mu(x),
$$
for all Borel sets $E\subset [0,+\infty)$. If $t=1$, then define $\mu_1 \in \calM^+$ by 
$$
\mu_1:=\mu(\{0\}) \cdot \delta.
$$
\end{definition}

\noindent We provide motivation for Definition~\ref{definition_14_may_2021_20:29} 
 in Remark~\ref{motivation_definition_14_may_2021_20:29}. 
Our first result is the following.

\begin{theorem}
\label{theorem_11_may_2021_20:17}
Let $A$ be a Banach subalgebra of $\calM^+$ $($with the induced norm$),$ 
such that it has the following property$:$
$$
\textrm{\em(}\mathbf{P}\textrm{\em) For all }\mu \in A, \textrm{\em and  all } 
t\in [0,1], \;\mu_t \in A.
$$
Then for all $n\in \mN,$ $\eSL_n(A)=\eE_n(A)$. 
\end{theorem}

\goodbreak 

\section{Proof of Theorem~\ref{theorem_11_may_2021_20:17}} 
\label{section_3}

\noindent For $\mu \in \calM^+$, it can be seen that $\|\mu_t\|\leq \|\mu\|$ for all $t\in [0,1]$. 
Also, $\delta_t=\delta$ for all $t\in [0,1]$.  For $\mu,\nu\in \calM^+$, and $t\in [0,1]$, we have 
 \begin{equation}
 \label{eq_11_may_2021_20:23}
 (\mu+\nu)_t=\mu_t+\nu_t.
 \end{equation}
 The following two results were shown in \cite{Sas}, but we include their proofs for convenience, 
 and to keep the discussion self-contained. 
 
 \begin{lemma}
 \label{lemma_11_may_2021_20:24}
 Let $\mu,\nu \in \calM^+$ and $t\in [0,1]$. Then 
 $$
 (\mu\ast \nu)_t =\mu_t \ast \nu_t.
 $$
 \end{lemma}
 \begin{proof}
  If $E$ is a Borel subset of $[0,+\infty)$, then
$$
(\mu \ast \nu)_t (E)
=
\int_E \!\!(1-t)^x \;\!d(\mu \ast \nu)(x)
=\!\!
\displaystyle
\iint_{
\substack{x+y \;\!\in \;\!E \\  x,y\;\!\in\;\! [0,+\infty) }}\!\!
(1-t)^{x+y} \;\!d\mu(x) \;\!d \nu(y).
$$
On the other hand,
\begin{eqnarray*}
\phantom{AAA}
( \mu_t \ast \nu_t )(E)
\!\!\!\!&=&\!\!\!\!
\int_{y \;\!\in\;\! [0,+\infty)}\!\! \mu_t (E-y) \;\!d \nu_t(y)
\\
\!\!\!\!&=&\!\!\!\!
\int_{y\;\! \in\;\! [0,+\infty)}\!\!
\Big( \int_{\substack{x \;\!\in \;\!E-y  \\ x \;\!\in \;\![0,+\infty)}}
\!\!(1-t)^x\;\! d \mu(x) \Big)
d \nu_t (y)
\\
\!\!\!\!&=&\!\!\!\!
\int_{y\;\! \in\;\! [0,+\infty)}\!\!
(1-t)^y\Big( \int_{\substack{x \;\!\in \;\!E-y  \\ x \;\!\in \;\![0,+\infty)}}\!\!
(1-t)^x\;\! d \mu(x) \Big)
d \nu (y)
\\
\!\!\!\!&=&\!\!\!\!\!\!
\iint_{\substack{x+ y \;\!\in\;\! E \\ x,y \;\!\in\;\! [0,+\infty)}}\!\!
(1-t)^{x+y} \;\!d \mu(x)\;\! d \nu (y).\phantom{AAAAAAAAA}\qedhere
\end{eqnarray*}
 \end{proof}
 
\noindent We will also use the following consequence  of the Radon-Nikodym theorem 
(see for example \cite[Theorem~6.12]{Rud}): If $\mu \in \calM^+$, 
then there exists `polar decomposition' of $\mu$, that is, 
there exists a measurable function $h$ such that $|h(x)|=1$ for all $x\in [0,+\infty)$, and $d\mu=h\;\!d|\mu|$. 

\begin{lemma}
\label{lemma_11_may_2021_20:28}
Let $\mu\in \calM^+$ and $t_0\in [0,1]$. Then 
$$
\lim_{t\;\!\rightarrow\;\! t_0}\mu_t =\mu_{t_0}.
$$
\end{lemma}
\begin{proof} $\;$

\noindent $1^\circ$
Consider first the case when $t_0 \in [0,1)$. 
Given an $\epsilon>0$,  let $R>0$ be 
large enough so that $|\mu|((R,+\infty))<\epsilon$. 
Let $t \in [0,1)$. There exists a 
Borel measurable function $w$ such that 
$$
d(\mu_t-\mu_{t_0})(x)=e^{-\;\!iw(x)}d| \mu_t-\mu_{t_0}|(x).
$$ 
Thus
\begin{eqnarray*}
\|\mu_t -\mu_{t_0}\|
\!\!\!\!&=& \!\!\!\!
|\mu_t -\mu_{t_0}|([0,+\infty))
=
\int_{[0,+\infty)} \!\!e^{iw(x)} d(\mu_t-\mu_{t_0})(x)
\\
\!\!\!\!&=&\!\!\!\! 
\Big|\int_{[0,+\infty)} \!\!e^{iw(x)} d(\mu_t-\mu_{t_0})(x)\Big|
\\
\!\!\!\!&=&\!\!\!\!
\Big| \int_{[0,+\infty)} \!\!e^{iw(x)} \Big( (1-t)^x - (1-t_0)^x \Big) \;\!
d\mu(x) \Big|.
\end{eqnarray*}
Hence
\begin{eqnarray*}
\|\mu_t -\mu_{t_0}\|  
\!\!\!\!&\leq &\!\!\!\!
\Big| \int_{[0,R]} \!\!e^{iw(x)} \Big( (1-t)^x \!- (1-t_0)^x \Big)\;\! d\mu(x) \Big|
    \\
  \!\!\!\!  & &\!\!\!\!
    +
    \Big| \int_{(R,+\infty)} \!\!e^{iw(x)} \Big( (1-t)^x \!- (1-t_0)^x \Big)\;\! d\mu(x) \Big|
    \\
  \!\!\!\!  &\leq &\!\!\!\!
    \max_{x\;\!\in\;\! [0,R] }\Big| (1-t)^x \!- (1-t_0)^x \Big| |\mu|([0,R])
    +
    2 |\mu|((R,+\infty))
    \\
 \!\!\!\!   &\leq &\!\!\!\!
     \max_{x\;\!\in \;\![0,R] }\Big| (1-t)^x \!- (1-t_0)^x \Big| |\mu|([0,+\infty))
    +
    2 \epsilon.
\end{eqnarray*}
But by the mean value theorem applied to the function 
$t\mapsto (1-t)^x$,
$$
(1-t)^x \!- (1-t_0)^x=-(t-t_0) x (1-c)^{x-1}=-(t-t_0) x \frac{(1-c)^{x}}{1-c},
$$
for some $c$ (depending on $x$, $t$ and $t_0$) in between $t$ and $t_0$.
Since $c$ lies between $t$ and $t_0$, and since both
$t$ and $t_0$ lie in $[0,1)$, and $x\in [0,R]$, it follows that $(1-c)^x \leq 1$ and
$$
\frac{1}{1-c} \leq\max\Big\{ \frac{1}{1-t}, \frac{1}{1-t_0}\Big\}.
$$
Thus using the above, and the fact that $x\in [0,R]$,
\begin{eqnarray*}
\max_{x\;\!\in\;\! [0,R] }\Big| (1-t)^x\! - (1-t_0)^x \Big|
\!\!\!\!&=&\!\!\!\!
\max_{x\;\!\in\;\! [0,R] }|t-t_0| x (1-c)^x\frac{1}{1-c}\\
\!\!\!\!& \leq&\!\!\!\! |t-t_0|\cdot R \cdot 1 \cdot \max\Big\{ \frac{1}{1-t}, \frac{1}{1-t_0}\Big\}.
\end{eqnarray*}
Hence we have
\begin{eqnarray*}
\!\!\!\!&&\!\!\!
\limsup_{t \;\!\rightarrow \;\!t_0}
\Big(\max_{x\;\!\in \;\![0,R] }\Big| (1-t)^x - (1-t_0)^x \Big| 
|\mu|([0,+\infty))\Big)
\\
\!\!\!\!&\leq&\!\!\!
\limsup_{t\;\! \rightarrow\;\! t_0}
\Big(|t-t_0| \cdot R\cdot \max\Big\{ \frac{1}{1-t}, \frac{1}{1-t_0}\Big\} \cdot|\mu|([0,+\infty))\Big)
\\
\!\!\!\!&=&\!\!\!
0 \cdot R \cdot \frac{1}{1-t_0}|\mu|([0,+\infty))
=0.
\end{eqnarray*}
Consequently,
$$
\limsup_{t \;\!\rightarrow \;\!t_0} \|\mu_t -\mu_{t_0}\| \leq 2\epsilon.
$$
But the choice of $\epsilon>0$ was arbitrary, and so
$$
\limsup_{t \;\!\rightarrow \;\!t_0} \|\mu_t -\mu_{t_0}\|=0.
$$
Since $\|\mu_t -\mu_{t_0}\|\geq 0$, we conclude that
$$
\lim_{t \;\!\rightarrow \;\!t_0} \|\mu_t -\mu_{t_0}\|=0.
$$

\noindent $2^\circ$ Now let us consider the case when $t_0=1$. Assume initially that
$\mu(\{0\})=0$. We will show that
$$
\lim_{t\;\!\rightarrow \;\!1} \mu_t=0
$$
in $\calM^+$.  Given an $\epsilon>0$, first choose a $r>0$ small
enough so that $|\mu|([0,r])<\epsilon$. This is possible, since
$\mu(\{0\})=0$.  There exists a Borel measurable function $w$ such
that $ d\mu_t(x)= e^{- \;\! iw(x)}d| \mu_t|(x)$.  Thus
  \begin{eqnarray*}
    \|\mu_t\|
    \!\!\!&=&\!\!\! |\mu_t |([0,+\infty))
    =
    \int_{[0,+\infty)} \!\!\!e^{iw(x)}\;\! d\mu_t(x)
    \\
   \!\!\! &=&\!\!\!
    \int_{[0,+\infty)}\!\!\! e^{iw(x)}(1-t)^x  d\mu(x)
    =
    \Big| \int_{[0,+\infty)}\!\!\! e^{iw(x)} (1-t)^x  d\mu(x) \Big|
    \\
   \!\!\! &\leq &\!\!\!
    \Big| \int_{[0,r]} \!\!\!e^{iw(x)} (1-t)^x  d\mu(x) \Big|
    +
    \Big| \int_{(r,+\infty)} \!\!\!e^{iw(x)} (1-t)^x d\mu(x) \Big|
    \\
   \!\!\! &\leq &\!\!\!
    |\mu|([0,r])
    +
    (1-t)^r |\mu|((r,+\infty))\phantom{\Big| \int_{(r,+\infty)} }
    \\
    \!\!\!&\leq &\!\!\!
    \epsilon+
    (1-t)^r |\mu|([0,+\infty)).\phantom{\Big| \int_{(r,+\infty)} }
\end{eqnarray*}
Consequently,
$\displaystyle 
\limsup_{t \;\!\rightarrow \;\!1} \|\mu_t \| \leq \epsilon.
$ 
As $\epsilon>0$ was arbitrary, 
$$
\limsup_{t \;\!\rightarrow\;\! 1} \|\mu_t \|=0.
$$
Since $\|\mu_t \|\geq 0$, we conclude that
$\displaystyle 
\lim_{t \;\!\rightarrow \;\!1} \|\mu_t \|=0.
$ 

Finally, if $\mu(\{0\})\neq 0$, then define
 $
\nu:= \mu- \mu(\{0\})\delta \in \calM^+.
$ 
It is clear that $\nu(\{0\})=0$ and $\nu_t=\mu_t-
\mu(\{0\})\delta$.  Since
$$
\displaystyle 
\lim_{t \;\!\rightarrow \;\!1} \nu_t =0,
$$
we obtain
 $\displaystyle 
\lim_{t\;\! \rightarrow \;\!1} \mu_t =\mu(\{0\})\delta
$ 
in $\calM^+$.
\end{proof}

\noindent In the final section, we will use the Laplace transform of measures $\mu \in \calM^+$:
For $\mu \in \calM^+$, we define for $s\in \mC_{{\scaleobj{0.81}{\geq 0}}}:=\{s\in \mC:\textrm{Re}(s)\geq 0\}$,
$$
\widehat{\mu}(s)=\int_{[0,+\infty)} \!\!\!e^{-sx} d\mu(x).
$$
Let $d\mu =e^{iw(x)} d|\mu|(x)$ for a measurable function $w$. For all $s\in  \mC_{{\scaleobj{0.81}{\geq 0}}}$,
$$
|\widehat{\mu}(s)|
\!=\!
\Big| \int_{[0,\infty)} \!\!\!\!e^{-sx} d\mu(x)\Big| 
\!=\!
\Big|\int_{[0,\infty)} \!\!\!\!e^{-\;\! iw(x)} e^{-sx} d|\mu|(x) \Big| 
\!\leq \!
 \int_{[0,\infty)} \!\!\!\!1\;\! d|\mu|(x) 
 \!=\!
 \|\mu\|.
$$
For $t\in [0,1)$, $s\in  \mC_{{\scaleobj{0.81}{\geq 0}}}$ and $\mu \in \calM^+$, $\widehat{\mu_t} (s)=\widehat{\mu}(s-\log(1\!-t))$. Indeed,
$$
\widehat{\mu_t} (s)
\!=\!
\int_{[0,\infty)} \!\!\!\!e^{-sx} (1\!-t)^x d\mu(x)
\!=\!
\int_{[0,\infty)} \!\!\!\!e^{-sx}e^{x\log (1-t)} d\mu
\!=\!
\widehat{\mu}(s-\log(1\!-t)).
$$
Recall that $\mu_1=\mu(\{0\})\delta$. As 
$\widehat{\delta}(s)=1$ for all $s\in \mC_{{\scaleobj{0.81}{\geq 0}}}$, 
$\widehat{\mu_1}(s)=\mu(\{0\})$. As with the Laplace transforms of ${\textrm{L}}^1$ functions, the same proof, mutatis mutandis, shows 
$\widehat{\mu\ast \nu}(s)=\widehat{\mu}(s)\;\! \widehat{\nu}(s)$ for $s\in \mC_{{\scaleobj{0.81}{\geq 0}}}$ and $\mu,\nu \in \calM^+$. 

We provide some motivation for Definition~\ref{definition_14_may_2021_20:29} in the following remark.

\begin{remark}
\label{motivation_definition_14_may_2021_20:29} 
We want to create a homotopy taking  $\mu\in \SL_n(\calM^+)$ to $I_n$ 
(with all diagonal entries equal to $\delta$). 
In the case of disc algebra, one uses dilations $f\mapsto f(t\;\!\cdot)$, $t\in [0,1]$ to take 
$I_n$ to $f\in \SL_n(A(\mD))$. 
Motivated by this, and bearing in mind that the Laplace transform $\widehat{\mu}$ of $\mu\in \calM^+$ 
satisfies the following analogue of the Riemann-Lebesgue lemma (we include a proof of this  in the Appendix) 
$$
\lim_{\mR\;\!\owns \;\!s\;\!\rightarrow \;\!+\infty} \widehat{\mu}(s)=\mu(\{0\}),
$$
it is natural to try and construct a homotopy by `translating' the Laplace transform and then taking the inverse Laplace transform. As 
$$
\lim_{t\;\!\rightarrow \;\!1-} -\log(1-t)=+\infty, 
$$ 
we try 
$$
H(t):=\mu_t:= \textrm{Inverse Laplace transform of }\pmb{(}s\mapsto \widehat{\mu}(s-\log(1-t))\pmb{)}.
$$
To determine $\mu_t$, we write 
$$
\int_{[0,+\infty)} \!\!\!\!e^{-(s-\log(1-t))x} d\mu(x)
\!=\!
\int_{[0,+\infty)} \!\!\!\!e^{-s x} (1\!-\!t)^x d\mu(x) 
\!=\!
\int_{[0,+\infty)} \!\!\!\!e^{-s x} d\mu_t(x),
$$
where $\mu_t$ is as in Definition~\ref{definition_14_may_2021_20:29}. \hfill$\Asterisk$
\end{remark}

 
 \noindent Finally, we are ready to prove Theorem~\ref{theorem_11_may_2021_20:17}. 
 
 \begin{proof}(of Theorem~\ref{theorem_11_may_2021_20:17}): 
 We need to show $\SL_n(A)\subset \E_n(A)$. Let
 $$
 M=\left[\!\!\begin{array}{ccc} \mu_{11} &\cdots & \mu_{1n}\\ 
 \vdots & \ddots & \vdots \\
 \mu_{n1} & \cdots & \mu_{nn} \end{array}\!\! \right]\in \SL_n(A).
 $$
 For $t\in [0,1]$, define 
 $$
 H(t)=M_t=\left[\!\!\begin{array}{ccc} (\mu_{11})_t &\cdots & (\mu_{1n})_t\\ 
 \vdots & \ddots & \vdots \\
 (\mu_{n1})_t & \cdots & (\mu_{nn})_t \end{array} \!\!\right].
 $$
 Thanks to \eqref{eq_11_may_2021_20:23} and Lemma~\ref{lemma_11_may_2021_20:24}, 
 as $M\in \SL_n(A)$, we have $M_t\in \SL_n(A)$ for all $t\in [0,1]$, 
 since 
 $$
 \delta=\delta_t=(\det M)_t= \det (M_t).
 $$ 
 The continuity of $[0,1]\owns t\mapsto M_t \in \SL_n(A)$ follows from 
 Lemma~\ref{lemma_11_may_2021_20:28} and Lemma~\ref{lemma_14_june_2021_10:24}.  We note that $M_1= C \delta$, 
 where $C\in \SL_n(\mC)$ is the constant matrix 
 $$
 C=\left[\!\! \begin{array}{ccc}
 \mu_{11}(\{0\}) & \cdots &  \mu_{1n}(\{0\})\\
 \vdots & \ddots & \vdots \\
  \mu_{n1}(\{0\}) & \cdots &  \mu_{nn}(\{0\})
 \end{array}\!\!\right] .
 $$
 But $\SL_n(\mC)$ is path connected, and so there exists a homotopy, say $h:[0,1]\rightarrow \SL_n(\mC)$, 
 taking $C$ to the identity matrix $I_n\in \SL_n(\mC)$. Combining $H$ with $h$, 
 we obtain a homotopy $\widetilde{H}:[0,1]\rightarrow \SL_n(A)$ that takes $M$ to $I_n\in \SL_n(A)$:
 $$
 \widetilde{H}(t)=\left\{\begin{array}{ll} H(2t) & \textrm{if }\; t\in \left[0,\displaystyle
 {\scaleobj{0.81}{\frac{1}{2}}}\right],\\[0.3cm]
 h(2t-\!1) \delta & \textrm{if }\; t\in \left[{\scaleobj{0.81}{\displaystyle\frac{1}{2}}},1\right].\end{array}\right.
 $$
 So $M\in \SL_n(A)$ is null-homotopic. By Proposition~\ref{prop_11_may_2021_18:36}, $M\in \E_n(A)$. 
 \end{proof}
 
 \section{Examples}
 \label{section_5}
 
\noindent Using Theorem~\ref{theorem_11_may_2021_20:17}, and with auxiliary results,  we give many illustrative examples of classical Banach algebras $A$ for which $\SL_n(A)=\E_n(A)$.
 
 \subsection{The measure algebra $\calM^+$} $\;$
 
 \noindent 
 Trivially, the full algebra $\calM^+$  has  the property ($\mathbf{P}$). So for all $n\in \mN$, $\SL_n(\calM^+)=\E_n(\calM^+)$.
 
 \goodbreak
 
 \subsection{The Wiener-Laplace algebra $\delta \;\!\mC+{\textrm{\normalfont L}}^1[0,+\infty)$} $\;$
 
 \noindent 
 Consider the Wiener-Laplace algebra $\calW^+$ of the half plane,
  of all functions defined in the half plane $\mC_{{\scaleobj{0.81}{\geq 0}}}:=\{s\in \mC:\textrm{Re}(s)\geq 0\}$ that
  differ from the Laplace transform of an ${\textrm{L}}^{1}[0,+\infty)$ function
  by a constant. 
  
  \goodbreak 
  
  \noindent The Wiener-Laplace algebra $\calW^+$ is a Banach algebra with pointwise operations, 
  and the norm $
  \|\widehat{f}+\alpha\|_{\calW^+}=\|f\|_{1}+|\alpha|, $ for all $\alpha \in \mC$ and all $f \in
  {\textrm{L}}^{1}[0,+\infty)$,  
  where 
  $$
  \|f\|_1:=\int_0^{+\infty}\!\! |f(x)|dx.
  $$
  Then $\calW^+$ is precisely the set of Laplace transforms of
  elements of the Banach subalgebra $A=\delta \;\!\mC+{\textrm{L}}^1[0,+\infty) $ of $\calM^+$ consisting of all complex
  Borel measures $\mu= \mu_a+\alpha\;\! \delta$, where $\mu_a$ is
  absolutely continuous (with respect to the Lebesgue measure) and
  $\alpha \in \mC$. Recall that $\mu_a\in \calM^+$ is {\em absolutely continuous}  
  if there exists an $f\in {\textrm{L}}^1[0,+\infty)$ such that, with the 
  Lebesgue measure on $[0,+\infty)$ denoted by $dx$, we have  
  $$
  \mu_a=\int_{{\scaleobj{2}{\cdot}}}f\;\!dx,
  $$
   that is, 
   $$
    \mu_a(E)=\int_E f(x) \;\! dx \textrm{ for all Borel }E\subset [0,+\infty).
  $$
  The $\calM^+$-norm of $\mu\in A\!=\!\delta \;\!\mC\!+\!{\textrm{L}}^1[0,+\infty)\subset \calM^+$ 
  is 
  $$
  \|\mu\|\!=\!\|f\|_{1}\!+\!|\alpha|
  $$ 
  for 
  $$
  \mu=\alpha\;\! \delta+\int_{{\scaleobj{2}{\cdot}}}f\;\!dx.
  $$
    This Banach subalgebra $A=\delta \;\!\mC+{\textrm{L}}^1[0,+\infty)$ of $\calM^+$ has the property ($\mathbf{P}$), and so for all $n\in \mN$, we have 
  \begin{eqnarray*}
   \SL_n(\delta \;\!\mC+{\textrm{L}}^1[0,+\infty))\!\!\!\!&=&\!\!\!\!
   \E_n(\delta \;\!\mC+{\textrm{L}}^1[0,+\infty)), \textrm{ and}\\
   \SL_n(\calW^+)\!\!\!\!&=&\!\!\!\! \E_n(\calW^+).
  \end{eqnarray*}

 \subsection{Measures without a singular nonatomic part} $\;$
 
 \noindent If $\lambda\geq 0$, then we use the notation $\delta_{\{\lambda\}}\in \calM^+$ to denote the Dirac measure with support $\{\lambda\}$, that is, for all Borel subsets $E\subset [0,+\infty)$,
$$
\delta_{\{\lambda\}}(E)=\left\{\begin{array}{ll} 1 & \textrm{if }\; \lambda\in E,\\
0 & \textrm{if }\; \lambda\not\in E.
\end{array}\right.
$$
(So $\delta=\delta_{\{0\}}$.) 
Define the subalgebra $\calA^+$ of $\calM^+$ consisting of all complex Borel measures that do not have a singular
  non-atomic part. 
  
  \noindent The algebra $\calA^+$ is the set of all $\mu\in \calM^+$ that have a 
  decomposition 
  $$
 \mu={\scaleobj{0.9}{\int_{{\scaleobj{0.72}{\bullet}}}}} \;\!f\;\!dx +\!\!\!\!\!\!\!\!\!\!\!\!\!\!\!\!
 {\scaleobj{0.9}{\sum_{\phantom{AAAAA}0=\lambda_0<\lambda_1,\lambda_2,\lambda_3,\cdots}}}
 \!\!\!\!\!\!\!\! f_n \delta_{\{\lambda_n\}},
$$
where 
\begin{eqnarray*} 
&& f\in {\textrm{L}}^1[0,\infty),\\
&& dx\textrm{  is the Lebesgue measure on }[0,+\infty),\\
&&  \lambda_n\in [0,+\infty) \textrm{ for all }n\in \mZ_{{\scaleobj{0.72}{\geq 0}}}:=\{0,1,2,3,\cdots\},\\
&& (f_n)_{n\geq 0}\textrm{ is an absolutely summable sequence of complex numbers,}\\
&& \phantom{ (f_n)_{n\geq 0}} \textrm{ i.e., }
\|(f_n)_{n\geq 0}\|_1:=\!
{\scaleobj{0.9}{\sup_{\substack{F\textrm{ finite}\\
F\;\!\subset\;\!\mZ_{{\scaleobj{0.81}{\geq 0}}}}} \;\sum_{n\;\!\in\;\! F} }} \;|f_n|<\infty.
\end{eqnarray*}
 The $\calM^+$-norm of $\mu\in \calA^+\subset \calM^+$ reduces to 
$$
\|\mu\|=\|f\|_{1} +\|(f_n)_{n\geq 0}\|_1\textrm{ for } 
 \mu=
 {\scaleobj{0.9}{\int_{{\scaleobj{0.71}{\bullet}}}}}\;\! f\;\!dx+\!\!\!\!\!\!\!\!\!\!\!\!\!\!\!\!
 {\scaleobj{0.9}{\sum_{\phantom{AAAAA}0=\lambda_0<\lambda_1,\lambda_2,\lambda_3,\cdots}}}
 \!\!\!\!\!\!\!\!  f_n \delta_{\{\lambda_n\}}\in \calA^+.
 $$
 The Banach subalgebra $\calA^+$ of $\calM^+$ possesses the property ($\mathbf{P}$). 
 Hence $\SL_n(\calA^+)=\E_n(\calA^+)$ for all $n\in \mN$. 
 We remark that the Bass stable rank of $\calA^+$ is $\infty$ 
 (see \cite{MikSas}). 
 
 \subsection{The analytic almost periodic Wiener algebra $\textrm{\normalfont APW}^+$} $\;$
 
 \noindent Define the subalgebra $\calA^+_0$ of $\calM^+$ to be the set of all 
 measures $\mu\in \calM^+$ that are of the form 
 $$
 \mu=\!\!\!\!\!\!\!\!\!\!\!\!\!\!\!\!
 {\scaleobj{0.9}{\sum_{\phantom{AAAAA}0=\lambda_0<\lambda_1,\lambda_2,\lambda_3,\cdots}}}
 \!\!\!\!\!\!\!\! \!\!\!\!\!\!f_n \delta_{\{\lambda_n\}},
 $$
 where the sequence of coefficients $(f_n)_{n\geq 0}$ is absolutely summable. 
 The  $\calM^+$-norm of $\mu\in \calA_0^+\subset \calM^+$ is 
 given by 
 $ \|\mu\|=\|(f_n)_{n\geq 0}\|_1$ for $\mu \in \calA_0^+$. 
 The algebra $\calA_0^+$ is a Banach subalgebra of $\calM^+$, 
 and has the property ($\mathbf{P}$). 
 Hence for all $n\in \mN$, $\SL_n(\calA^+_0)=\E_n(\calA^+_0)$. 
 
 Recall the classical Banach algebra $\textrm{APW}^+$ of almost periodic functions $f$, 
 whose Bohr-Fourier coefficients $(f_\lambda)_{\lambda \in \mR}$, are summable,  
 and $f_\lambda=0$ for $\lambda<0$. The algebra operations in $\textrm{APW}^+$ are defined pointwise,  and the  norm is given by  
 $$
 \|f\|_{1}=
 {\scaleobj{0.9}{\sum_{\lambda\geq 0}}}\; |f_\lambda|.
 $$
 (See  the following Subsection~\ref{subsection_12_may_2021_18:06} for an explanation of the key terms.) 
 The Banach algebra $\calA_0^+$ is isomorphic as a Banach algebra to $\textrm{APW}^+$ via the isomorphism 
 $$
 \calA_0^+\owns \mu=\!\!\!\!\!\!\!\!\!\!\!\!\!\!\!\!\!\!\!\!\!
 {\scaleobj{0.9}{\sum_{\phantom{AAAAAA}0=\lambda_0<\lambda_1,\lambda_2,\lambda_3,\cdots}}}
 \!\!\!\!\!\!\!\!\! \!\! f_n \delta_{\{\lambda_n\}}\;\;\mapsto 
\!\!\! \!\!\!\!\!\!\!\!\!\!\!\!
{\scaleobj{0.9}{\sum_{\phantom{AAAAAA}0=\lambda_0<\lambda_1,\lambda_2,\lambda_3,\cdots}}}
\!\!\!\!\!\!\!\!\!\!\!\!\!\! f_n e^{i \lambda_n x} \in \textrm{APW}^+.
 $$
 So for all $n\in \mN$, $\SL_n({\textrm{APW}}^+)=\E_n({\textrm{APW}}^+)$. 

 \goodbreak 
 
 \subsection{The algebra $\textrm{\normalfont AP}^+$ of analytic almost periodic functions} $\;$
 \label{subsection_12_may_2021_18:06} 
 
 \noindent We refer the reader to \cite{Cor} and \cite{BKS} for details on almost periodic functions. 
 For $\lambda \in \mR$, let  $\bbe_\lambda:=e^{i\lambda\;\!\cdot }\in {\textrm{L}}^\infty(\mR)$. 
 Let $\calT$ be the space of trigonometric polynomials, i.e., 
 $\calT$ is the linear span of $\{\bbe_\lambda: \lambda \in \mR\}$. 
 Define $\textrm{AP}$ to be the closure of $\calT$ with respect to the ${\textrm{L}}^\infty(\mR)$-norm $\|\cdot\|_\infty$. Then 
 $\textrm{AP}$ is a Banach algebra with pointwise operations, and the norm $\|\cdot\|_\infty$. 
 The space $\calT$ is equipped with
  the inner product 
 $$
 \langle p,q\rangle =\lim_{R\;\!\rightarrow \;\!+\infty} \frac{1}{2R}\int_{-R}^R p(x)\;\! \overline{q(x)} \;\!dx \quad (p,q\in \calT),
 $$
  where $\overline{\cdot}$ denotes complex conjugation. 
 The limit exists since
$$
\langle \bbe_\lambda, \bbe_{\lambda'}\rangle =\left\{ \begin{array}{ll} 1 & \textrm{if }\; \lambda=\lambda',\\
0 & \textrm{if }\;\lambda \neq \lambda'.
\end{array}\right.
$$
For $\lambda \in \mR$ and $f\in \textrm{AP}$, the {\em Bohr-Fourier coefficient} $f_\lambda$ is defined as follows: 
If $(p_n)_n$ is a sequence in $\calT$ converging to $f$ in $\textrm{AP}$, then 
$$
f_\lambda=\lim_{n\;\!\rightarrow \;\!\infty} \langle p_n,\bbe_{-\lambda}\rangle.
$$  
Define the {\em Bohr spectrum of $f$} to be the set 
 $
\sigma(f)=\{\lambda\in \mR: f_\lambda \neq 0\},
$ 
which is at most countable. 
Let $\textrm{AP}^+$ be the subspace of $\textrm{AP}$ given by 
$$
\textrm{AP}^+=\{f\in \textrm{AP}:\sigma(f)\subset [0,+\infty)\}.
$$
Each $f\in \textrm{AP}^+$ has a holomorphic extension to the upper half-plane
$$
\mU:=\{z\in \mC:\textrm{Im}\;\!z>0\}.
$$ 
The set $\textrm{APW}^+$ is a dense subset of $\textrm{AP}^+$ (since the trigonometric polynomials 
are dense in $\textrm{AP}^+$).  
We will use this, and the following consequence Proposition~\ref{prop_11_may_2021_18:36}, 
 to show that for all $n\in \mN$ we have 
$\SL_n({\textrm{AP}}^+)=\E_n({\textrm{AP}}^+)$. We remark that the Bass stable rank of ${\textrm{AP}}^+$ is $\infty$; see \cite{MorRup}. We have the following result:

\begin{theorem}$\;$
\label{theorem_13_may_2021_14:08}

\noindent 
Let $A$ be a commutative unital Banach algebra. 
Let $S$ be a Banach algebra containing the unit of $A$ such that 
$$
\begin{array}{l}
\bullet\; \textrm{for all }n\in \mN, \;\eSL_n(S)=\eE_n(S),\\
\bullet\; S\textrm{ is a full subalgebra of }A,\;(\textrm{i.e.},\textrm{ if }S\cap  \textrm{\em GL}_1(A)=  \textrm{\em GL}_1(S)),\\
 \bullet\; \textrm{the inclusion map is continuous},\textrm{ and}\\ 
\bullet\;S\textrm{ is dense in }A.  
\end{array}
$$
Then for all $n\in \mN,$ $\eSL_n(A)=\eE_n(A)$.
\end{theorem}

\goodbreak

\begin{proof}
Let $f\in \SL_n(A)$. As $S$ is dense in $A$, we can find an element $g \in S^{n\times n}$ such that 
$$
\|g-f\|<\frac{1}{\|f^{-1}\|}.
$$
Then for all $t\in [0,1]$, 
$$
\|t f^{-1}(g-f)\|<1\cdot \|f^{-1}\|\cdot \frac{1}{\|f^{-1}\|} =1.
$$
So $I_n+tf^{-1}(g-f)\in \textrm{GL}_n(A)$. As 
 $$
g=f+g-f=f(I_n+f^{-1}(g-f)),
$$  
$g\in \textrm{GL}_n(A)$. Thus $\det g \in S\cap \textrm{GL}_1(A)$, giving $\det g\in \textrm{GL}_1(S)$. Also, 
$(I_n+t f^{-1}(g-f))^{-1} g \in \textrm{GL}_n(A)$ for all $t\in [0,1]$, which implies 
$$
\Delta(t):=\det ((I_n+t f^{-1}(g-f))^{-1} g)\in \textrm{GL}_1(A)
$$
for all $t\in [0,1]$. For $t\in [0,1]$, define the matrix $H(t)\in A^{n\times n}$ to be 
the matrix obtained by scaling (any one, say) the first column of $(I_n+t f^{-1}(g-f))^{-1} g$ 
by $(\Delta(t))^{-1}$. Then 
\begin{eqnarray*}
\det (H(t))\!\!\!\!&=&\!\!\!\! (\Delta(t))^{-1} \det ((I_n+t f^{-1}(g-f))^{-1} g)\\
\!\!\!\!&=&\!\!\!\!(\Delta(t))^{-1} \Delta(t)=1.
\end{eqnarray*}
Hence $H(t)\in \SL_n(A)$, and the map $[0,1]\owns t\mapsto H(t)\in \SL_n(A)$ is continuous. 
We have $H(1)=f$.  Also, $H(0)$ is the matrix obtained by scaling the first column 
of $g$ by $(\det g)^{-1}$. So $H(0)\in \SL_n(S)=\E_n(S)$. By Proposition~\ref{prop_11_may_2021_18:36}, 
there exists a homotopy $h:[0,1]\rightarrow \SL_n(S)$ such that $h(0)=I_n$ and $h(1)=H(0)$. 
Thanks to the fact that the inclusion map for $S\subset A$ is continuous, $h:[0,1]\rightarrow \SL_n(A)$ is also continuous. 
By combining $h$ and $H$, we get a homotopy $\widetilde{H}:[0,1]\rightarrow \SL_n(A)$ 
such that $\widetilde{H}(0)=I_n$ and $\widetilde{H}(1)=H(1)=f$. By  Proposition~\ref{prop_11_may_2021_18:36}, $f\in \E_n(A)$. 
Consequently, $\SL_n(A)=\E_n(A)$ for all $n\in \mN$.
\end{proof}

\begin{corollary}
For all $n\in \mN,$ $\eSL_n({\textrm{\em AP}}^+)=\eE_n({\textrm{\em AP}}^+)$.
\end{corollary}
\begin{proof} For $f\in \textrm{APW}^+$, we have 
$\|f\|_\infty\leq \|f\|_1$, showing that the inclusion map is continuous.  
We have $\textrm{APW}^+$ is a dense subset of ${\textrm{AP}}^+$. 
The full subalgebra assumption can be seen to hold by using the corona theorem  
for ${\textrm{APW}}^+$ (see e.g. \cite[Theorem~2.4]{RodSpi}). 
Moreover, we have seen in the previous subsection that 
$\SL_n({\textrm{APW}}^+)\!=\!\E_n({\textrm{APW}}^+)$ for all $n\in \mN$.
\end{proof}

\goodbreak

\subsection{${\textrm{\normalfont APW}}^+_S$ and ${\textrm{\normalfont AP}}^+_S$} 

\noindent Recall that an {\em additive sub-semigroup of the group $(\mR,+)$} 
is a subset $S\subset \mR$ with the properties $0\in S$, and $\lambda+\lambda' \in S$ whenever $\lambda, \lambda' \in S$. 
Given an additive sub-semigroup $S\subset [0,\infty)$, 
we define the Banach subalgebra ${\textrm{APW}}^+_S$ of ${\textrm{APW}}^+$ by 
$$
{\textrm{APW}}^+_S:=\{ f\in {\textrm{APW}}^+: \sigma(f)\subset S\},
$$
with the induced norm $\|\cdot\|_1$ from ${\textrm{APW}}^+$. Similarly, 
we define the Banach subalgebra ${\textrm{AP}}^+_S$ of ${\textrm{AP}}^+$ by 
$$
{\textrm{AP}}^+_S:=\{ f\in {\textrm{AP}}^+: \sigma(f)\subset S\},
$$
with the induced norm $\|\cdot\|_\infty$ from ${\textrm{AP}}^+$. Thus if $S=[0,\infty)$, 
the ${\textrm{APW}}^+_S$ and ${\textrm{AP}}^+_S$ are 
${\textrm{APW}}^+$ and ${\textrm{AP}}^+$, respectively. 

\noindent The Banach algebra ${\textrm{APW}}^+_S$ is isomorphic to the 
Banach subalgebra $\calA_{0,S}^+$ of $\calM^+$ consisting of all measures 
$$
\mu=
\!\!\!\!\!\!\!\!\!\!\!\!\!\!\!\!\!\!\!\!\!\sum_{\phantom{AAAAAA} \lambda_0,\lambda_1,\lambda_2,\lambda_3,\cdots \;\!\in\;\! S}
 \!\!\!\!\!\!\!\!f_n \delta_{\{\lambda_n\}},
$$
where the sequence of coefficients $(f_n)_{n\geq 0}$ is absolutely summable. 
The Banach subalgebra $\calA_{0,S}^+$ of $\calM^+$  has the property (${\mathbf{P}}$). 
So for all $n\in \mN,$ $\SL_n(\calA_{0,S}^+)=\E_n(\calA_{0,S}^+)$, and 
 $\SL_n({\textrm{APW}}_S^+)=\E_n({\textrm{APW}}_S^+)$. 

To get the result for ${\textrm{AP}}_S^+$, we again use Theorem~\ref{theorem_13_may_2021_14:08}. 
The algebra ${\textrm{APW}}_S^+$ is a dense subset of ${\textrm{AP}}_S^+$. Next we show that 
${\textrm{APW}}_S^+$ is a  full subalgebra of ${\textrm{APW}}_S^+$. Let $f\in {\textrm{APW}}_S^+$ 
be such that there exists a $g\in {\textrm{AP}}_S^+$ such that $f\cdot g=\mathbf{1}$, 
where $\mathbf{1}$ is the constant function taking value $1$ everywhere on $\mR$. Taking the Gelfand transform, we obtain 
$$
\widehat{f}\cdot \widehat{g}=\mathbf{1},
$$
 where $\mathbf{1}$ is the constant function taking value $1$ everywhere on the 
 maximal ideal space $M( {\textrm{AP}}_S^+)$ of ${\textrm{AP}}_S^+$. 
 In particular, $\widehat{f}$ does not vanish anywhere on $M( {\textrm{AP}}_S^+)$. 
 By the Arens-Singer theorem \cite[Theorem~4.1]{AreSin} (see also \cite[Theorem~3.1]{Bot}), 
 the maximal ideal spaces of $ {\textrm{AP}}_S^+$ and $ {\textrm{APW}}_S^+$ are homemorphic. 
 Thus the Gelfand transform of $f\in {\textrm{APW}}_S^+$, namely 
$$
\widehat{f}:M({\textrm{APW}}_S^+)\rightarrow \mC,
$$
 is also nowhere zero on $M({\textrm{APW}}_S^+)$. By the elementary theory of 
 Banach algebras (see e.g. \cite[Theorem~2.7]{Gam}), it follows that $f$ is invertible as an element of ${\textrm{APW}}_S^+$. 
So ${\textrm{APW}}_S^+$ is a  full subalgebra of ${\textrm{APW}}_S^+$. 
By Theorem~\ref{theorem_13_may_2021_14:08}, 
we conclude that for all $n\in \mN,$ $\SL_n({\textrm{AP}}_S^+)=\E_n({\textrm{AP}}_S^+)$.

\subsection{Nonexample: $A\subset \calM^+$ failing ($\mathbf{P}$), 
but $\textrm{\normalfont SL}_n(A)=\textrm{\normalfont E}_n(A)$} $\;$

\noindent We show that while the property ($\mathbf{P}$) is sufficient for the 
Banach subalgebra $A\subset \calM^+$ for having $\SL_n(A)=\E_n(A)$ for all $n\in \mN$, it is not necessary. 
We do this by constructing a Banach subalgebra $A$ of $\calM^+$ for which the property ($\mathbf{P}$)  
does not hold, but for which $\SL_n(A)=\E_n(A)$ for all $n\in \mN$.  
This Banach subalgebra was also considered in \cite{Sas0.5}, but 
here we will show that its maximal ideal space has topological dimension $2$, 
which will be used to show $\SL_n(A)=\E_n(A)$ for all $n\in \mN$.

 Let $dx$ denote the Lebesgue measure, and let $\mu\in \calM^+$ be given by 
$$
\mu=\delta+{\scaleobj{0.9}{\int_{{\scaleobj{0.72}{\bullet}}} }}\;\!(2x-3)e^{-x} dx.
$$
Then for $s\in \mC_{{\scaleobj{0.81}{\geq 0}}}$, 
$$
\widehat{\mu}(s)=
{\scaleobj{0.9}{\frac{s(s-1)}{(s+1)^2}}}.
$$
For $m\in \mN$, we denote the $m$-fold convolution $\mu\ast \cdots \ast \mu$ by $\mu^{\ast m}$. 
Let $A\subset \calM^+$ be the Banach subalgebra of $\calM^+$ generated by $\delta, \mu$, that is, $A$ is the closure of all `polynomials' in $\mu$: 
$$
p=c_0 \delta+ c_1 \mu +c_2 \mu^{\ast 2}+\cdots +c_d \mu^{\ast d},
$$
where $d$ is any nonnegative integer, and $c_0,c_1,\cdots, c_d\in \mC$ are arbitrary. 

 Let us first show that $A$ does not have the property ($\mathbf{P}$). We show that for a certain $t$, $\mu_t \not \in A$. In fact, set $t=1-\frac{1}{e}\in (0,1)$.  Since the polynomials in $\mu$ are dense in $A$, given $\epsilon=\frac{1}{10}>0$, there exists a $d\in \mN$, and complex numbers $c_0,c_1,\cdots, c_d$, such that 
$$
\|\mu_t -(c_0 \delta+ c_1 \mu +c_2 \mu^{\ast 2}+\cdots +c_d \mu^{\ast d})\|<\epsilon.
$$
But if $\nu:=\mu_t -(c_0 \delta+ c_1 \mu +c_2 \mu^{\ast 2}+\cdots +c_d \mu^{\ast d})$, then 
for $s\in \mC_{{\scaleobj{0.81}{\geq 0}}}$, 
$$
|\widehat{\mu_t}(s) -(c_0 +c_1 \widehat{\mu}(s)+c_2 (\widehat{\mu}(s))^2+\cdots+
c_d (\widehat{\mu}(s))^d | 
=|\widehat{\nu}(s)|\leq \|\nu\|<\epsilon.
$$
We have $-\log(1\!-t)=-\log \frac{1}{e}=1$, so that the above gives 
$$
|\widehat{\mu}(s+1)-(c_0 +c_1 \widehat{\mu}(s)+c_2 (\widehat{\mu}(s))^2+\cdots+
c_d (\widehat{\mu}(s))^d | <\epsilon.
$$ 
Setting $s=0$, this gives 
$$
\Big| {\scaleobj{0.81}{\frac{(0+1)(0+1-1)}{(0+1+1)^2}}}-c_0\Big|
=
|c_0|<\epsilon. 
$$
On the other hand, with $s=1$, we obtain 
$$
\Big| {\scaleobj{0.81}{\frac{(1+1)(1+1-1)}{(1+1+1)^2}}}-c_0\Big|
=
\Big|{\scaleobj{0.81}{\frac{2}{9}}}-c_0\Big|<\epsilon. 
$$
Adding the last two inequalities gives the contradiction that 
$$
{\scaleobj{0.81}{\frac{2}{9}}}\leq \Big|{\scaleobj{0.81}{\frac{2}{9}}}-c_0\Big|+|c_0|<2\;\!\epsilon=
{\scaleobj{0.81}{\frac{2}{10}}}.
$$
Hence $A$ does not have the property ($\mathbf{P}$). 

\goodbreak 

\noindent Finally, we show that for the Banach algebra $A$, $\SL_n(A)=\E_n(A)$ for all $n\in \mN$. 
To do this, we will use the result from \cite[Theorem~1.1(3)]{Bru}, which says that if the topological dimension $\dim M(R)$ of the maximal ideal space $M(R)$ of a commutative unital complex Banach algebra $R$ is $0,1$ or $2$, then $\SL_n(R)=\E_n(R)$. We will show that the topological dimension of the maximal ideal space $M(A)$ of our Banach algebra $A$ is equal to $2$. (Recall that for a normal space $X$, $\dim X\leq d$ if every finite open cover of $X$ can be refined by an open cover whose order $\leq d+1$. If $\dim X\leq d$ and the statement $\dim X\leq d-1$ is false, then we say that the {\em topological dimension of $X$} is $\dim X=d$. Also, for a commutative complex semisimple unital Banach algebra $R$, the maximal ideal space $M(R)$ (set of all complex homomorphisms)  is equipped with the Gelfand topology, 
which is the induced subset topology from the dual space $\calL(R;\mC)$ equipped with the weak-$\ast$ topology.)
We will need the following result; see \cite[Theorem 1.4, page 68]{Gam}.

\begin{proposition} 
\label{prop_spec=j.spec}
  Let $B$ be a finitely generated Banach algebra$,$ generated by $b_1,
  \dots, b_m$. Then the joint spectrum of $b_1, \cdots, b_m$ in $B,$
  namely the set
 $$
\sigma_B(b_1, \cdots, b_m)=\{ (\widehat{b_1}(\varphi), \dots,
\widehat{b_m}(\varphi)):\varphi \in M(B)\}\;\;(\subset \mC^m),
$$ 
is homeomorphic to the maximal ideal space $M(B)$ of the Banach algebra $B$. 
 $($Here
$\widehat{\cdot}:B\rightarrow C(M(B);\mC)$ denotes the Gelfand transform.$)$
\end{proposition}

\noindent So in our case, it suffices to show that the joint spectrum of
$\delta$ and $\mu$ in $A$ has topological dimension $2$. We observe that
\begin{eqnarray}
\nonumber
\sigma_A( \delta, \mu)&=& \{(1, \widehat{\mu}(\varphi)): \varphi \in M(A)\}
= \{1\} \times \{\widehat{\mu}(\varphi): \varphi \in M(A)\}
\\
\label{eq_spec_1}
&=& \{1\} \times \sigma_A (\mu).
\end{eqnarray}
Hence it is enough to show that $\sigma_A (\mu)\subset \mC\simeq \mR^2$ has topological dimension equal to $2$. We recall the following result, which relates the
spectrum of an element $x$ of a subalgebra of a Banach algebra with
the spectrum of $x$ in the Banach algebra; see \cite[Theorem 10.18,
p. 238]{Rud2}.

\begin{proposition}
\label{prop_Rud}
  Let $B$ be a unital Banach algebra$,$ and $S$ be a Banach subalgebra
  of $B$ that contains the unit of $B$. If $x\in S,$ then
  $\sigma_S(x)$ is the union of $\sigma_B(x)$ and a $($possibly empty$)$
  collection of bounded components of the complement of $\sigma_B(x)$.
\end{proposition}

\noindent We apply this with $B\!:=\!
\mC\;\! \delta \!+\!{\textrm{L}}^{1}[0,+\infty)$ (the Wiener-Laplace algebra),  
$S\!:=\!A$, and $x\!:=\!\mu\in A \subset B
\!=\!
\mC\;\! \delta\! +\!{\textrm{L}}^{1}[0,+\infty)$. The maximal ideal space of
$B\!=\!\mC\;\! \delta \!+\!{\textrm{L}}^{1}[0,+\infty)$ can be identified (see e.g. \cite[pp.~112-113]{grs}) as a topological space with  
$$
\{s
\in (\mC\cup\{\infty\}) : \textrm{ and Re}(s) \geq 0\}.
$$ 
\noindent (Here  we identify $\mC\cup\{\infty\}$ with the Riemann sphere.) The identification is done via the Laplace transform: For 
$$
\nu=\alpha \;\!\delta + \displaystyle \int_{{\scaleobj{0.72}{\bullet}}} f\;\! dx \in \mC\;\! \delta +{\textrm{L}}^{1}[0,\infty),
$$
 and $s\in \mC_{{\scaleobj{0.81}{\geq 0}}}$, 
$$
\widehat{\nu}(s):= \alpha +\int_0^\infty f(x)e^{-sx} dx ,\quad \textrm{and}\quad  
 \widehat{\nu}(\infty):=\alpha.
 $$
  Consequently,
$$
\sigma_{B} (\mu) 
\!=\!
\Big\{ {\scaleobj{0.9}{ \frac{s (s\!-\!1)}{(s\!+\!1)^2}}} : s \in (\mC\cup\{\infty\} ),
  \textrm{ Re}(s)\! \geq \!0\Big\} 
 \!\stackrel{s=\frac{1+z}{1-z}}{=}\!
\Big\{ {\scaleobj{0.9}{\frac{z\!+\!z^2}{2}}} : |z| \!\leq \!1\Big\}.
$$
It can be shown that this last set is the interior of a
simple closed curve $C$ shown as the bold curve  in Figure~\ref{spec}. 
Thus the complement of $\sigma_{B} (\mu) $ has no bounded components. By Proposition~\ref{prop_Rud}
we conclude that $\sigma_{A}
(\mu)=\sigma_{B}
(\mu)$. So $\sigma_{A} (\mu)$ has topological dimension equal to $2$. 
From \eqref{eq_spec_1}, also $\sigma_A( \delta, \mu)= \{1\} \times \sigma_A (\mu)$. So $\sigma_A( \delta, \mu)$ has topological dimension equal to $2$ too.  By Proposition~\ref{prop_spec=j.spec},
$\dim M(A)$ is equal to $2$.

\begin{figure}[H]
    \center  
    \psfrag{m}[c][c]{${\scaleobj{0.93}{-1}}$}
    \psfrag{p}[c][c]{${\scaleobj{0.93}{1}}$}
    \includegraphics[width=6.3cm]{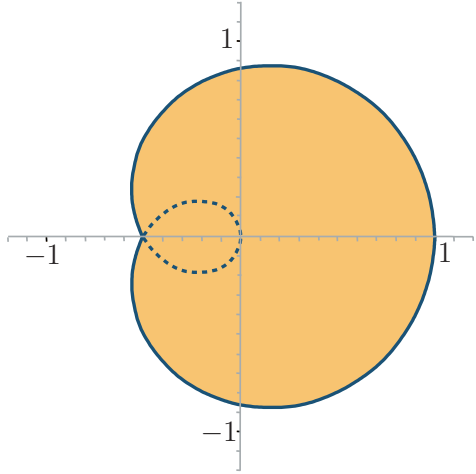}
     \caption{{\small The simple closed curve $C$ is depicted by the bold line.
      The bold curve $C$, with the dotted
      curve, is together the curve $\theta \mapsto
      (e^{i\theta}+e^{2i\theta})/2:[0,2\pi] \rightarrow \mC$.}}
    \label{spec}
\end{figure}

\subsection{Subalgebras $A_d$ of $\calM^+$ with $\dim M(A_d)=2d$} 
 In the above, we constructed a Banach sub-algebra $A$ of $\calM^+$ 
whose maximal ideal space $M(A)$ has topological dimension equal to $2$. 
Now, for each $d\in \mN$, we construct a Banach subalgebra $A_d$ of $\calM_+$ such that the  topological dimension of their maximal ideal space $M(A_d)$ is $2d$, and for which there also holds  $\SL_n(A_d)=\E_n(A_d)$ for all $n\in \mN$. 

\goodbreak

 \noindent Set $\mD=\{z\in \mC:|z|<1\}$. Let $\mD^{\;\!d}:=\mD\times \cdots \times \mD$ ($d$ times) be the $d$-dimensional polydisc in $\mC^d$. Every holomorphic $f:\mD^{\;\!d} \rightarrow \mC$ has a Taylor expansion, and we denote the Taylor coefficients of $f$ by $f_{\bbn}$: 
$$
f(\bbz)=\sum_{\bbn\;\!\in\;\! (\mZ_{{\scaleobj{0.81}{\geq 0}}})^d} \;\!f_{\bbn} \bbz^\bbn \quad 
(\bbz=(z_1,\cdots, z_d)\in \mD^{\;\!d}),
$$
where $\mZ_{{\scaleobj{0.72}{\geq 0}}}:=\{0,1,2,3,\cdots\}$, and 
we use the usual multi-index notation 
$$
\bbz^\bbn=z_1^{n_1}\cdots z_d^{n_d}
$$
 for $\bbz=(z_1,\cdots, z_d)\in \mD^d$ and $\bbn=(n_1,\cdots, n_d)\in (\mZ_{{\scaleobj{0.72}{\geq 0}}})^d$. 
The {\em Wiener algebra} is the set ${\textrm{W}}^+(\mD^d)$ of all holomorphic $f:\mD^d \rightarrow \mC$ 
such that its Taylor coefficients are summable:
$$
\|f\|_1:=\sum_{\bbn\;\!\in\;\! (\mZ_{{\scaleobj{0.81}{\geq 0}}})^d}\;\! |f_{\bbn}|<\infty.
$$
With pointwise operations, and the norm $\|\cdot\|_1$, ${\textrm{W}}^+(\mD^{\;\!d})$ is a Banach algebra. The maximal ideal space $M({\textrm{W}}^+(\mD^{\;\!d}))$ of ${\textrm{W}}^+(\mD^{\;\!d})$ is 
homeomorphic to the the closure $\overline{\mD^{\;\!d}}$ of $\mD^{\;\!d}$ in $\mC^d$ via the map 
$$
\overline{\mD^{\;\!d}} \owns \bbz \mapsto \pmb{(}\;\! {\textrm{W}}^+(\mD^{\;\!d})\owns f\mapsto f(\bbz)\in \mC\;\!\pmb{)}\in M({\textrm{W}}^+(\mD^{\;\!d})).
$$  
Hence $\dim (M({\textrm{W}}^+(\mD^{\;\!d})))=2d$. 

 Let the set $B=\{e_i: i\in I\}$ be a Hamel basis for the vector space $\mR$ over $\mQ$ 
(with vector space addition and scalar multiplication given by the usual arithmetic operations). 
We can always replace $e_i$ by $-e_i$ to ensure that all the $e_i$ are strictly positive. 
For $d\in \mN$, take any distinct $e_{i_1},\cdots, e_{i_d}$ belonging to $B$, and set 
$$
S=\{n_1 e_{i_1}+\cdots+ n_d e_{i_d}: n_1,\cdots, n_d\in \mZ_{{\scaleobj{0.72}{\geq 0}}}\}\subset [0,\infty).
$$
Then $S$ is an additive sub-semigroup of the group $(\mR,+)$, and the Banach algebra ${\textrm{APW}}_S^+$ is isomorphic as a Banach algebra to ${\textrm{W}}^+(\mD^{\;\!d})$ via the map 
$$
{\textrm{W}}^+(\mD^d)\owns f=\sum_{\bbn\in (\mZ_{{\scaleobj{0.81}{\geq 0}}})^d} f_{\bbn} \bbz^\bbn \mapsto 
\sum_{\bbn\in (\mZ_{{\scaleobj{0.81}{\geq 0}}})^d} f_{\bbn} e^{i (n_1 e_{i_1}+\cdots+ n_d e_{i_d}) x} \in {\textrm{APW}}_S^+.
$$
We have seen that the Banach algebra $\calA_{0,S}^+$ is isomorphic as a Banach algebra to ${\textrm{APW}}_S^+$. We take $A_d=\calA_{0,S}^+$.  We have 
$$
\dim (M(\calA_{0,S}^+))=\dim (M({\textrm{APW}}_S^+))=\dim (M({\textrm{W}}^+(\mD^{\;\!d})))=2d.
$$ 
We have already seen that 
for all $n\in \mN,$ $\SL_n(\calA_{0,S}^+)=\E_n(\calA_{0,S}^+)$. 
 
\goodbreak 

\section*{Appendix} 

\begin{proposition}
Let $\mu \in \calM^+$. Then 
 $$
\displaystyle \lim_{\mR\;\! \owns\;\! s\;\!\rightarrow \;\!+\infty} \widehat{\mu}(s)=\mu(\{0\}).
$$ 
\end{proposition}

\goodbreak 

\begin{proof} 
First suppose that $\mu$ satisfies $\mu(\{0\})=0$. Let $\epsilon>0$. There exists an $r>0$ such that $|\mu|([0,r])<\epsilon$, thanks to the assumption that $\mu(\{0\})=0$. Let $w$ be a Borel measurable function such that $\mu$ has the polar decomposition $d\mu(x)=e^{-iw(x)} d|\mu|(x)$. Then 
\begin{eqnarray*}
|\widehat{\mu}(s)|
\!\!\!\!&=&\!\!\!\! \Big| \int_{[0,+\infty)} \!\!e^{-sx} e^{-iw(x)} d|\mu|(x)\Big|\\
\!\!\!\!&\leq &\!\!\!\! \Big| \int_{[0,r]} \!\!e^{-sx} e^{-iw(x)} d|\mu|(x)\Big| 
+\Big|\int_{(r,+\infty)} \!\!e^{-sx} e^{-iw(x)} d|\mu|(x)\Big|\\
\!\!\!\!&\leq &\!\!\!\! \int_{[0,r]} \!\!1\;\!d|\mu|(x)+e^{-sr}\int_{(r,+\infty)}\!\! 1\;\!d|\mu|(x)\leq \epsilon +e^{-sr} |\mu|((r,+\infty)).
\end{eqnarray*}
So $\displaystyle \limsup_{s\;\!\rightarrow \;\!+\infty} |\widehat{\mu}|\leq \epsilon$. As $\epsilon$ was arbitrary, 
 $
 \displaystyle
 \limsup_{s\;\!\rightarrow \;\!+\infty} |\widehat{\mu}|\leq 0.
 $
 As $|\widehat{\mu}(s)|\geq 0$, 
 $$
 \displaystyle \lim_{s \;\!\rightarrow \;\!+\infty}|\widehat{\mu}(s)|=0, \textrm{ and  }  
 \displaystyle 
 \lim_{s \;\!\rightarrow \;\!+\infty}\widehat{\mu}(s)=0.
 $$ 
 If $\mu(\{0\})\neq 0$, we complete the proof by considering $\nu:=\mu-\mu(\{0\})\delta$, which satisfies 
 $\nu(\{0\})=0$ and $\widehat{\nu}=\widehat{\mu}-\mu(\{0\})$.
\end{proof}

\end{document}